\documentclass[sn-mathphys,Numbered]{sn-jnl}

\usepackage{graphicx}%
\usepackage{multirow}%
\usepackage{amsmath,amssymb,amsfonts}%
\usepackage{amsthm}%
\usepackage{mathrsfs}%
\usepackage[title]{appendix}%
\usepackage{xcolor}%
\usepackage{textcomp}%
\usepackage{manyfoot}%
\usepackage{booktabs}%
\usepackage{algorithm}%
\usepackage{algorithmicx}%
\usepackage{algpseudocode}%
\usepackage{listings}%

\theoremstyle{thmstyleone}%
\newtheorem{theorem}{Theorem}
\newtheorem{proposition}[theorem]{Proposition}%

\theoremstyle{thmstyletwo}%
\newtheorem{example}{Example}%
\newtheorem{remark}{Remark}%
\newtheorem{lemma}{Lemma}%
\theoremstyle{thmstylethree}%
\newtheorem{definition}{Definition}%
\begin{document}

\title[Article Title]{Euler-Maruyama approximation for stochastic fractional neutral integro-differential equations with weakly singular kernel}
\author[1]{\fnm{Javad} \sur{A. Asadzade}}\email{javad.asadzade@emu.edu.tr}

\author[2]{\fnm{Nazim} \sur{I. Mahmudov}}\email{nazim.mahmudov@emu.edu.tr}
\equalcont{These authors contributed equally to this work.}

\affil[1]{\orgdiv{Department of Mathematics}, \orgname{Eastern Mediterranean University}, \orgaddress{\street{} \city{Mersin 10, 99628, T.R.}, \postcode{5380},\country{North Cyprus, Turkey}}}

\affil[2]{\orgdiv{Department of Mathematics}, \orgname{Eastern Mediterranean University}, \orgaddress{\street{} \city{Mersin 10, 99628, T.R.}, \postcode{5380},\country{North Cyprus, Turkey}}}

\affil[2]{\orgdiv{	Research Center of Econophysics}, \orgname{Azerbaijan State University of Economics (UNEC)}, \orgaddress{\street{Istiqlaliyyat Str. 6}, \city{Baku }, \postcode{1001},  \country{Azerbaijan}}}

\abstract{This manuscript examines the problem of nonlinear stochastic fractional neutral integro-differential equations with weakly singular kernels. Our focus is on obtaining precise estimates to cover all possible cases of Abel-type singular kernels. Initially, we establish the existence, uniqueness, and continuous dependence on the initial value of the true solution, assuming a local Lipschitz condition and linear growth condition. Additionally, we develop the Euler-Maruyama method for numerical solution of the equation and prove its strong convergence under the same conditions as the well-posedness. Moreover, we determine the accurate convergence rate of this method under global Lipschitz conditions and linear growth conditions.}

\keywords{Fractional stochastic neutral integro-differensial equations, weakly singular kernels, local Lipschitz condition well-posedness, Euler–Maruyama approximation.}
\pacs[MSC Classification]{65C30, 65C20, 65L05, 65L20}

\maketitle

\tableofcontents

\section{Introduction}\label{sec1}

Integro-differential equations find widespread applications across various scientific domains, including biological population dynamics (\cite{16,17,37}), wave propagation phenomena (\cite{22}), and reactor dynamics (\cite{23}). The increasing development and deeper understanding of fractional calculus have led to the emergence of fractional integro-differential equations in fields such as electromagnetic wave modeling (\cite{16}), population systems (\cite{27,43}), and other areas. These equations provide powerful tools for capturing complex dynamics characterized by memory effects, non-local interactions, and stochastic influences, making them indispensable in modern scientific research.

Moreover, to account for pervasive noise factors in real-world scenarios, stochastic integro-differential equations have been introduced in anomalous diffusion processes (\cite{7}), stochastic feedback systems (\cite{32}), and option pricing models (\cite{9}). Nowadays, there is growing interest among scholars in stochastic fractional equations (\cite{28,39,47}) due to their applicability in investigating memory, hereditary, and hidden properties of noisy systems in physics (\cite{24}), mathematical finance (\cite{38}), and ecological epidemiology (\cite{45}). These equations play a crucial role in understanding and modeling complex phenomena influenced by both fractional calculus and stochastic processes, offering insights into the intricate interplay between deterministic and random dynamics.

For stochastic fractional integro-differential equations (SFIDEs) with regular kernels, the work by Badr and El-Hoety (\cite{13}) initially addressed the well-posedness of the linear case. Subsequently, various numerical methods have been explored, including block pulse approximation (\cite{20}), Galerkin methods (\cite{30}), spectral collocation methods (\cite{35}), operational matrix methods (\cite{29}), and meshless collocation methods (\cite{29}). However, the presence of weak singularities in fractional derivatives and the low regularity of stochastic noise pose significant challenges in concrete analysis. Dealing with singular integral kernels, particularly in the context of stochastic integrals, adds further complexity to the analysis. (For more information see, (\cite{1}-\cite{51}))

\bigskip

In this article, we focus on the initial value problem of $d$-dimensional nonlinear stochastic fractional integro-differential equations (SFNIDEs) in Ito's sense. This problem arises from actual situations where fractional integro-differential equations with Abel-type singular kernels are used to model certain phenomena with non-negligible noise sources. The problem is defined as follows:
\bigskip

Consider the following initial value problem of $d$-dimensional nonlinear SFNIDEs in Ito's sense (\ref{A}), stimulated by the non-negligible noise source of some actual problems modeled by fractional integro-differential equations with Abel-type singular kernels \cite{8}.
\begin{align}\label{A}
	{^{C}} D^{\alpha}\bigg\lbrace z(t)-\sum_{i=1}^{n} I^{\alpha_{i}}f_{i}(t,z(t))\bigg\rbrace=&g_{0}(t,z(t))+\int_{0}^{t}\frac{g_{1}(t,s,z(s))}{(t-s)^{\beta_{1}}}ds
	+\int_{0}^{t}\frac{g_{2}(t,s,z(s))}{(t-s)^{\beta_{2}}}dW{s}
\end{align}

for $t\in [0,T]$ with $z(0)=z_{0}\in R^{d}$. Where, $T>0$ is a given real number, ${^{C}}D^{\alpha}$ is the Caputo fractional derivative of order $0<\alpha\leq\alpha_{i}\leq 1,\quad i=1,\dots,n$, $\beta_{1}\in(0,1)$, $\beta_{2}\in (0,1/2)$,  and $W(t)$ denotes an r-dimensional standard Wiener process (i.e., Brownian motion) was defined on the complete probability space $(\Omega, \mathcal{F}, P)$ with a filtration $\lbrace \mathcal{F}_{t}\rbrace_{t\geq0}$ satisfying the usual conditions (i.e., it is right continuous and $\mathcal{F}_{0}$ contains all the P-null sets), and the initial value $z_{0}$ is an $\mathcal{F}_{0}$-measurable random variable defined on the same probability space such that $E[\vert z_{0}\vert^{p}]< +\infty $, for some integer $p \geq 2$.

This article is devoted to the following main goals:
\bigskip

$\bullet$
By satisfying the local Lipschitz condition and the linear growth condition, we establish the existence and uniqueness of the true solution to SFIDE (\ref{A}) in Theorem \ref{t.1}. Additionally, we demonstrate the continuous dependence on the initial value, addressing the well-posedness gap in Theorem \ref{t.2}.

\bigskip
\begin{theorem}\label{t.1}
	The Stochastic Fractional Neutral Integro-differential Equation  (\ref{A}) has a unique solution $z(t)$ under Assumptions (1), (3), and (4). Furthermore, for each positive $p\geq 2$,
	\begin{align*}
		E\Big[\vert z(t)\vert^{p}\Big]<+\infty, \quad \forall t\in[0,T].
	\end{align*}
\end{theorem}
\bigskip
\begin{theorem}\label{t.2}
	The solution of SFNIDE  (\ref{A}) is contionusly depends on the initial value in the mean square sense under Assumptions (1),(3), and (4).
\end{theorem}

\bigskip

$\bullet$ Under the same conditions that ensure well-posedness, the strong convergence of the Euler–Maruyama (EM) method is examined. Furthermore, by applying the global Lipschitz condition and linear growth condition, the precise convergence rate is determined, demonstrating the method's effectiveness in solving SFNIDE (\ref{A}).

\section{Preliminary}\label{sec2}
In this paper, we will use the following conventions unless stated otherwise. The expectation corresponding to a probability distribution $P$ will be denoted by $E$. If $A$ is a vector or matrix, its transpose will be represented by $A^{T}$. The notation $\vert\cdot\vert$ will be used to denote both the Euclidean norm on $R^{d}$ and the trace norm on $R^{d\times r}$. In other words, if $x \in R^{d}$, $\vert x\vert$ will refer to the Euclidean norm, and if $A$ is a matrix, $\vert A\vert$ will represent the trace norm. The indicator function of a set $S$ will be denoted by $1_{S}$, where $1_{S}(x) = 1$ if $x \in S$ and $0$ otherwise. For two real numbers $a$ and $b$, we will use the notation $a \lor b := \max(a, b)$ and $a \land b := \min(a, b)$. Additionally, the uppercase letter $C$ (with or without subscripts) will be used to represent a positive constant whose value may vary depending on its context, but it will always be independent of the step size $h$. Finally, we will introduce four mild assumptions that will be used later for the nonlinear functions $f_{i}$ $(i=1,\dots,n)$ and $g_{j}$ $(j = 0, 1, 2)$.
\subsection{Assumptions}

	\textbf{1)}  $\exists$ $L_{1}>0$ such that $\forall t_{1},t_{2},s\in [0,T]$ and
$\forall z\in R^{d},g_{1}$ and$g_{2}$ satisfy the condition: 
$$\vert g_{j}(t_{1},s,z)-g_{j}(t_{2},s,z)\vert\leq L_{1}(1+\vert z\vert)\vert t_{1}-t_{2}\vert,\quad j=1,2.$$
\textbf{2)}    $\exists$ $L_{2}>0$ such that   $\forall t,s_{1},s_{2}\in [0,T]$ and $\forall z\in R^{d}, g_{0},g_{1}$,$g_{2}$ and $f_{i}$ for $i=1,\dots,n$\\
satisfy the condition:
\begin{equation*}
	\begin{cases}
		\vert g_{0}(s_{1},z)-g_{0}(s_{2},z)\vert \lor\vert g_{j}(t,s_{1},z)-g_{j}(t,s_{2},z)\vert\leq L_{2}(1+\vert z\vert)\vert s_{1}-s_{2}\vert,\quad j=1,2,\\
		\vert f_{i}(s_{1},z)-f_{i}(s_{2},z)\vert \leq L_{2}(1+\vert z\vert)\vert s_{1}-s_{2}\vert,\quad i=1,2,\dots,n.
	\end{cases}
\end{equation*}
\textbf{3)}  $\forall$ integer $m\geq1$, $\exists K_{m}\geq 0$ depending only on $m$, such that $\forall t,s\in [0,T]$ and $\forall z_{1},z_{2}\in R^{d}$ with $\vert z_{1}\vert \lor \vert z_{2}\vert \leq m$, $g_{0},g_{1}, g_{2}$ and, $f_{i}$ for $i=1,\dots,n$ hold the local Lipsichitz condition:
\begin{equation*}
	\begin{cases}
		\vert g_{0}(s,z_{1})-g_{0}(s,z_{2})\vert \lor\vert g_{j}(t,s,z_{1})-g_{j}(t,s,z_{2})\vert\leq K_{m}\vert z_{1}-z_{2}\vert,\quad j=1,2,\\
		\vert f_{i}(s,z_{1})-f_{i}(s,z_{2})\vert \leq K_{m}\vert z_{1}-z_{2}\vert,\quad i=1,2,\dots,n.
	\end{cases}
\end{equation*}
\textbf{4)} $\exists$ $L>0$ such that   $\forall t,s\in [0,T]$ and $\forall z\in R^{d}, g_{0},g_{1}$,$g_{2}$, and $f_{i}$ for $i=1,\dots,n$ satisfy the linear growth condition: 
\begin{equation*}
	\begin{cases}
		\vert g_{0}(s,z)\vert \lor\vert g_{i}(t,s,z)\vert\leq L(1+\vert z\vert),\quad i=1,2.,\\
	\vert f_{i}(s,z)\vert\leq L(1+\vert z\vert),\quad for \quad i=1,2,\dots, n.
	\end{cases}
\end{equation*}

\begin{remark}\label{r.1}
We emphasize that the local Lipschitz condition mentioned above (i.e., Assumption 3) is weaker than the next global Lipschitz condition in order to represent the generality of our conclusions:  $\exists K>$ such that $\forall t, s \in [0,T]$ and $\forall z_{1},z_{2} \in R^{d}, g_{0}, g_{1}$, $g_{2}$ and $f_{i}$$(i=1,\dots,n)$ satisfy the inequalities
		\begin{equation}\label{4}
			\begin{cases}
				\vert g_{0}(s,z_{1})-g_{0}(s,z_{2})\vert \lor \vert g_{j}(t,s,z_{1})-g_{j}(t,s,z_{2})\vert\leq K\vert z_{1}-z_{2}\vert, \quad j=1,2.,\\
				\vert f_{i}(s,z_{1})-f_{i}(s,z_{2})\vert\leq K\vert z_{1}-z_{2}\vert, \quad i=1,2,\dots,n.
			\end{cases}
		\end{equation}
\end{remark}
\subsection{Some auxiliary results.}
In this part, we mention important information that deals with the elements of the fractional analysis, and some necessary lemmas that will use the proof of the theorem. 

\begin{itemize}
	\item Gamma function:
	\[
	\Gamma(\alpha)=\int_{0}^{\infty}t^{\alpha-1}e^{-t}dt,\quad\alpha>0.
	\]
	\item Beta function: 
	\[
	B(x,y)=\int_{0}^{1}t^{x-1}(1-t)^{y-1}dt,\quad x,y>0.
	\]
	\\
	Let $0<\alpha\leq 1$, $0\leq T<+\infty$ and $f : [0,T]\to R^{d}$ be a measurable function  and satisfies  $\int_{0}^{T} \vert f(s)\vert ds<+\infty$.
	\\
	\item Riemann-Liouville fractional integral:
	\[
	\mathcal{I}_{0+}^{\alpha}f(x)=\frac{1}{\Gamma(\alpha)}\int_{0}^{t}(t-s)^{\alpha-1}f(s)ds,\quad \forall t\in [0,T].
	\]
	\item Caputo derivative:
	\[
	{^{C}}D_{0+}^{\alpha}f(t)=D^{\alpha}f(t)=(\mathcal{I}^{1-\alpha}Df)(t), 
	\]
	where $D=\frac{d}{dt}$ is the usual derivative.
	\end{itemize}
\begin{proposition}\label{p.3}(\cite{51})
	The Riemann-Liouville fractional integral operator and the Caputo fractional derivative have the following properties:
	\begin{align*}
		&1. \mathcal{I}^{\alpha}(D^{\alpha}f(t))=f(t)-f(0),\\
		&2. D^{\alpha}(C)=0, \quad where\quad C>0,\\
		&3. D^{\alpha}(\mathcal{I}^{\alpha}f(t))=f(t).
	\end{align*}
\end{proposition}

The integral form of the SFNIDE (\ref{A}) defines it rigorously, similar to how fractional SDEs are defined.
\begin{align}\label{99}
	z(t)=&z_{0}+\sum_{i=1}^{n}\frac{1}{\Gamma(\alpha_{i})}\int_{0}^{t}\frac{f_{i}(\tau,z(\tau))}{(t-\tau)^{1-\alpha_{i}}}d\tau+\frac{1}{\Gamma(\alpha)}\int_{0}^{t}\frac{g_{0}(\tau,z(\tau))}{(t-\tau)^{1-\alpha}}d\tau\nonumber\\
	+&\frac{1}{\Gamma(\alpha)}\int_{0}^{t}(t-\tau)^{\alpha-1}\bigg(\int_{0}^{\tau}\frac{g_{1}(\tau,s,z(s))}{(\tau-s)^{\beta_{1}}}ds\bigg)d\tau\\
	+&\frac{1}{\Gamma(\alpha)}\int_{0}^{t}(t-\tau)^{\alpha-1}\bigg(\int_{0}^{\tau}\frac{g_{2}(\tau,s,z(s))}{(\tau-s)^{\beta_{2}}}dW(s)\bigg)d\tau\nonumber
\end{align}
Subsequently, we define the unique solution to SFNIDE (\ref{A}), using the form of (\cite{4}, Definition 2.1 of Chapter 2). 
\bigskip
\begin{definition}\label{d.1}
An $R^{d}$-valued stochastic process $\lbrace z(t)\rbrace_{t\in [0,T]}$ is called a solution to the SFNIDE (\ref{A}) if it has the following properties:

\bigskip

1) $\lbrace z(t)\rbrace_{t\in [0,T]}$ is continious and $\mathcal{F}_{t}$-adapted,

\bigskip

2) $f_{i}\in L([0,T]\times R^{d}, R^{d})$, $i=1,\dots,n$, $g_{0}\in L([0,T]\times R^{d}, R^{d})$, $g_{1}\in L(\lbrace (t,s): $

\bigskip

$\quad  0\leq s\leq t\leq T\rbrace \times R^{d}, R^{d})$ and, $g_{2}\in L(\lbrace (t,s): 0\leq s\leq t\leq T\rbrace \times R^{d}, R^{d\times r})$,

\bigskip

3) The equation (\ref{99}) almost surely holds $\forall t\in [0,T]$.\\
\end{definition} 
A solution of $\lbrace z(t)\rbrace$ is said to be unique if any other solution $\lbrace\bar{z}(t)\rbrace$ is indistinguishable from $\lbrace z(t)\rbrace$, that is
\begin{align*}
	P\big\lbrace z(t)=\bar{z}(t) \quad \forall t\in [0,T]\big\rbrace=1.
\end{align*}

We express the useful connection between SFNIDEs and SVIEs with the following theorem of this section.  

\begin{theorem}\label{t.5} Assume that $0<\alpha\leq \alpha_{i}\leq1,\quad \beta_{1}\in (0,1)$ and, $\beta_{2}\in (0,\frac{1}{2})$. Hence, a continuous and $\mathscr{F}_{t}$-adapted stochastic process $z(t)$ is a solution to the stochastic fractional integro-differential equation (\ref{A}), which is strictly defined with integral form (\ref{99}), if and only if it is solution of the following stochastic Volterra integral equation
	\begin{align}\label{789}
		z(t)=&z_{0}+\sum_{i=1}^{n}\int_{0}^{t}F_{i}(t,s,z(s))ds+\int_{0}^{t}G_{0}(t,s,z(s))ds\nonumber\\
		+&\int_{0}^{t}G_{1}(t,s,z(s))ds+\int_{0}^{t}G_{2}(t,s,z(s))dW(s),
	\end{align}
	where
	\begin{align*}
		&F_{i}(t,s,z(s))=\frac{1}{\Gamma(\alpha_{i})}f_{i}(s,z(s))(t-s)^{\alpha_{i}-1}, \quad for\quad i=1,\dots n\\
		&G_{0}(t,s,z(s))=\frac{1}{\Gamma(\alpha)}g_{0}(s,z(s))(t-s)^{\alpha-1}\\
		&G_{i}(t,s,z(s))=\frac{1}{\Gamma(\alpha)}\int_{s}^{t}(t-\tau)^{\alpha-1}(\tau-s)^{-\beta_{j}} g_{j}(\tau,s,z(s))d\tau\\
		&=\frac{(t-s)^{\alpha-\beta_{j}}}{\Gamma(\alpha)}\int_{0}^{1}(1-u)^{\alpha-1}u^{-\beta_{j}}g_{j}((t-s)u+s,s,z(s))du,\quad for \quad j=1,2.
	\end{align*}
\end{theorem}
\begin{proof} Assume that $z(t)$ is a solution of stochastic fractional integro-differensial equation (\ref{A}), $f_{2}\in L^{2}(\lbrace (t,s): 0\leq s\leq t\leq T\rbrace \times R^{d}; R^{d\times r})$, therefore equation (\ref{99})  almost surely satisfies. Hence, it follows from (\cite{1}, theorem 4.33 or 5.10) that
	\begin{align*}
		\int_{0}^{t}(t-\tau)^{\alpha-1}\bigg(\int_{0}^{\tau}(\tau-s)^{-2\beta_{2}}E\bigg[\vert f_{2}(\tau,s,z(s))\vert^{2}\bigg]ds\bigg)^{1/2}d\tau<+\infty,\quad for \quad all \quad t\in [0,T],
	\end{align*}
	which represents that the sufficent condition of stochastic Fubini theorem (\cite{2}, theorem 4.33 or 5.10) is held. Therefore, by using the theorem of Fubini to  equation (\ref{99}), we have 
	\begin{align*}
		z(t)=&z_{0}+\sum_{i=1}^{n}\frac{1}{\Gamma(\alpha_{i})}\int_{0}^{t}\frac{f_{i}(s,z(s))}{(t-s)^{1-\alpha_{i}}}ds+\frac{1}{\Gamma(\alpha)}\int_{0}^{t}\frac{g_{0}(s,z(s))}{(t-s)^{1-\alpha}}ds\\
		+&\frac{1}{\Gamma(\alpha)}\int_{0}^{t}\bigg(\int_{s}^{t}\frac{g_{1}(\tau,s,z(s))}{(t-\tau)^{1-\alpha}(\tau-s)^{\beta_{1}}}d\tau\bigg)ds\\
		+&\frac{1}{\Gamma(\alpha)}\int_{0}^{t}\bigg(\int_{s}^{t}\frac{g_{2}(\tau,s,z(s))}{(t-\tau)^{1-\alpha}(\tau-s)^{\beta_{2}}}d\tau\bigg)dW(s)
	\end{align*}
	which implies that $z(t)$ is a solution of the stochastic Volterra intergral equation (\ref{99}).
	
	Conversely, if $z(t)$ is a solution of the stochastic fractional integro-differential equation (\ref{A}), then $z(t)$ is also a solution of the stochastic Volterra integral equation (\ref{789}), according to the Fubini theorem.
\end{proof}
On the other hand  we will give the following lemmas that we can use to support our main results.
\begin{lemma}\label{l.1}(\cite{6})
	Let $a,b\in (0,1]$. Then for any $t\in [t_{n},t_{n+1}], n=1,\dots, N-1$, we have
	\begin{align*}
		\int_{0}^{t_{n}} \big\vert (t-\hat{s})^{a-b}-(t_{n}-\hat{s})^{a-b}\big\vert ds\leq Ch^{1\land (1+a-b)},\\
		\int_{0}^{t} \big\vert (t-\hat{s})^{a-b}-(t-s)^{a-b}\big\vert ds\leq Ch^{1\land (1+a-b)},
	\end{align*}
	where the positive constant $C$ depends on $T, a, b$, but not on $h$.
\end{lemma}
\begin{lemma}\label{l.2}(\cite{6})
	Let $a\in (0,1]$, $b\in (0,\frac{1}{2})$. Then $a-b\in (-\frac{1}{2}, 1)$, and for any $t\in [t_{n},t_{n+1}], n=1,\dots, N-1$, we have
	\begin{align*}
		\int_{0}^{t_{n}} \big\vert (t-\hat{s})^{a-b}-(t_{n}-\hat{s})^{a-b}\big\vert ds\leq \begin{cases}
			Ch^{2}, \quad if\quad a-b\in (\frac{1}{2},1),\\
			Ch^{2-\varepsilon},\quad if\quad a-b=\frac{1}{2},\\
			Ch^{1+2(a-b)},\quad if\quad a-b\in (-\frac{1}{2},\frac{1}{2}),
		\end{cases}
	\end{align*}
	\begin{align*}
		\int_{0}^{t_{n}} \big\vert (t-\hat{s})^{a-b}-(t_{n}-\hat{s})^{a-b}\big\vert ds\leq \begin{cases}
			Ch^{2}, \quad if\quad a-b\in (\frac{1}{2},1),\\
			Ch^{2-\varepsilon},\quad if\quad a-b=\frac{1}{2},\\
			Ch^{1+2(a-b)},\quad if\quad a-b\in (-\frac{1}{2},\frac{1}{2}),
		\end{cases}
	\end{align*}
	where the positive constant $C$ depends on $T, a, b, \varepsilon$, but not on $h$, and $\varepsilon\in (0,1)$.
\end{lemma}
\bigskip
$\bullet$  Young's inequality, a well-known result in mathematical analysis, asserts that for any pair of positive numbers $u$ and $v$ such that they fulfill the relationship $\frac{1}{u} + \frac{1}{v} = 1$, the inequality can be described in the following manner:   
\begin{align}\label{yon}
	ab\leq \frac{\delta a^{u}}{u}+\frac{b^{v}}{v\delta^{\frac{v}{u}}},\quad \forall a,b,u>0.
\end{align}
\section{Well-posedness of stochastic fractional neutral integro-differentional equation}\label{sec3}
We will examine the existence, uniqueness, and continuous dependence on the initial value of the exact solution to the stochastic fractional integro-differential equation (\ref{A}) using the preparation from the previous section.
\subsection{\textbf{Existence and uniqueness of solution of stochastic fractional neutral integro-differentional equation}}
We first propose the EM approximation to facilitate in the proof of the existence result. For each integer $N \geq 1$,  EM-approximation (\cite{3},\cite{4}) can be shown as 

\begin{align}\label{10}
	z^{N}(t)=&z_{0}+\sum_{i=1}^{n}\int_{0}^{t}F_{i}(t,s,\hat{z}^{N}(s))ds+\int_{0}^{t}G_{0}(t,s,\hat{z}^{N}(s))ds\nonumber\\
	+&\int_{0}^{t}G_{1}(t,s,\hat{z}^{N}(s))ds+\int_{0}^{t}G_{2}(t,s,\hat{z}^{N}(s))dW(s),
\end{align}
where the simplest step process $\hat{z}^{N}(t)=\sum_{n=0}^{N} z^{N}(t_{n})1_{[t_{n},t_{n+1}]}(t)$ and the mesh points $t_{n}=nh$ $(n=0,\dots,N)$ with $h=\frac{T}{N}$.	
\begin{lemma}\label{l.3}
If assumption (4) is true, then there is a positive constant $C$ that does not depend on the value of $N$. This constant satisfies the following inequalities for any integer $p\geq 2$, and for all values of $t$ in the interval $[0,T]$:
	\begin{align*}
		E\big[\vert z^{N}(t)\vert^{p}\big]\leq C\quad and,\quad E\big[\vert \hat{z}^{N}(t)\vert^{p}\big]\leq C,\quad \forall t\in [0,T].
	\end{align*}
\end{lemma}
\begin{proof}
	First, we prove the case where $p>2$. Let $k\geq1$ be an integer. We define the stopping time $$\rho^{N}_{m} = T\land \inf\lbrace t\in [0,T]: \vert z^{N}(t)\vert \geq m\rbrace,$$ where $\rho^{N}_{m}\uparrow T$ almost surely as $m\to \infty$. For convenience, we set $z^{N}_{m}(t)=z^{N}(t	\land \rho^{N}_{m})$ and $\hat{z}^{N}_{m}(t)=\hat{z}^{N}(t	\land \rho^{N}_{m})$ for all $t\in [0,T]$. Using Hölder inequality, Burkholder-Davis-Gundy inequality, assumption (4), and the fact that $E\big[\vert z_{0}\vert^{p}\big]<+\infty$, we can deduce from (\ref{10}) that there exists a positive constant $q\in (0,p\alpha)$, which only depends on $p>2$ and $0<\alpha\leq\alpha_{i}\leq1$, so that
	\begin{align*}
		&E\bigg[\big\vert z^{N}_{m}(t)\big\vert^{p}\bigg]\leq \frac{(n+4)^{p}}{\Gamma^{p}(\alpha)}\bigg\lbrace
		\Gamma^{p}(\alpha) E\big[\vert z_{0}\vert^{p}\big]+\sum_{i=1}^{n}\frac{\Gamma^{p}(\alpha) }{\Gamma^{p}(\alpha_{i})}E\bigg(\bigg\vert\int_{0}^{t	\land \rho^{N}_{m}}(t	\land \rho^{N}_{m}-s)^{\alpha_{i}-1}f_{i}(s,\hat{z}^{N}_{m}(s))ds\bigg\vert^{p}\bigg)\\
		&+E\bigg(\bigg\vert\int_{0}^{t\land\rho^{N}_{m}}(t	\land \rho^{N}_{m}-s)^{\alpha-1}g_{0}(s,\hat{z}^{N}_{m}(s))ds\bigg\vert^{p}\bigg)\\
		&+B^{p}(\alpha,1-\beta_{1}) E\bigg(\bigg\vert\int_{0}^{t\land\rho^{N}_{m}}(t	\land \rho^{N}_{m}-s)^{\alpha-\beta_{1}} \sup_{s\leq v\leq(t\land\rho^{N}_{m})}\vert g_{1}(v,s,\hat{z}^{N}_{m}(s))\vert ds\bigg\vert^{p}\bigg)\\
		&+B^{p}(\alpha,1-\beta_{2}) E\bigg(\bigg\vert\int_{0}^{t\land\rho^{N}_{m}}(t	\land \rho^{N}_{m}-s)^{2(\alpha-\beta_{2})} \sup_{s\leq v\leq(t\land\rho^{N}_{m})}\vert g_{2}(v,s,\hat{z}^{N}_{m}(s))\vert^{2} ds\bigg\vert^{\frac{p}{2}}\bigg)\bigg\rbrace\\
		&\leq C_{1}\bigg\lbrace 1+\sum_{i=1}^{n}\bigg(\int_{0}^{t\land \rho^{N}_{m}}(t\land \rho^{N}_{m}-s)^{\frac{p\alpha_{i}-q}{p-1}-1}ds\bigg)^{p-1} \int_{0}^{t\land \rho^{N}_{m}} (t\land \rho^{N}_{m}-s)^{q-1}\big(1+E\big[\vert \hat{z}^{N}_{m}(s)\vert^{p}\big]\big)ds\\
		&+2\bigg(\int_{0}^{t\land \rho^{N}_{m}}(t\land \rho^{N}_{m}-s)^{\frac{p\alpha-q}{p-1}-1}ds\bigg)^{p-1} \int_{0}^{t\land \rho^{N}_{m}} (t\land \rho^{N}_{m}-s)^{q-1}\big(1+E\big[\vert \hat{z}^{N}_{m}(s)\vert^{p}\big]\big)ds\\
		&+\bigg[\int_{0}^{t\land \rho^{N}_{m}}(t\land \rho^{N}_{m}-s)^{\frac{2(p\alpha-q)}{p-2}-1}ds\bigg]^{\frac{p-2}{2}} \int_{0}^{t\land \rho^{N}_{m}} (t\land \rho^{N}_{m}-s)^{q-1}\big(1+E\big[\vert \hat{z}^{N}_{m}(s)\vert^{p}\big]\big)ds\bigg\rbrace\\
		&\leq C_{2}\bigg(1+\int_{0}^{t\land \rho^{N}_{m}} (t\land \rho^{N}_{m}-s)^{q-1}\big(1+E\big[\vert \hat{z}^{N}_{m}(s)\vert^{p}\big]\big)ds\bigg)
	\end{align*}
 By taking the supremum on both sides of the equation, we find that the positive constants $C_{1}$ and $C_{2}$ are independent of $m$ and $N$. 
	
	\begin{align*}
		\sup_{0\leq \lambda\leq t}E\big[\vert z^{N}_{m}(\lambda)\vert^{p}\big]\leq& C_{2} \bigg\lbrace 1+ \sup_{0\leq \lambda\leq t} \int_{0}^{\lambda \land \rho^{N}_{m}} (\lambda\land \rho^{N}_{m}-s)^{q-1} \sup_{0\leq\eta\leq s}E\big[\vert z^{N}_{m}(\eta)\vert^{p}\big]ds\bigg\rbrace
	\end{align*}
	If we replace $v=\frac{s}{\lambda\land \rho^{N}_{m}}$, we will get 
	\begin{align*}
		\sup_{0\leq \lambda\leq t}E\big[\vert z^{N}_{m}(\lambda)\vert^{p}\big]\leq& C_{2}\bigg\lbrace 1+ \sup_{0\leq \lambda\leq t}(\lambda\land \rho^{N}_{m})^{q} \int_{0}^{1} (1-v)^{q-1} \sup_{0\leq\eta\leq (\lambda\land \rho^{N}_{m})v}E\big[\vert z^{N}_{m}(\eta)\vert^{p}\big]dv\bigg\rbrace\\
		\leq&C_{2}\bigg\lbrace 1+ t^{q}\int_{0}^{1} (1-v)^{q-1} \sup_{0\leq\eta\leq tv}E\big[\vert z^{N}_{m}(\eta)\vert^{p}\big]dv\bigg\rbrace
	\end{align*}
	When we alternate with $s=tv$, then we have 
	\begin{align*}
		\sup_{0\leq \lambda\leq t}E\big[\vert z^{N}_{m}(\lambda)\vert^{p}\big]
		\leq&C_{2}\bigg\lbrace 1+ \int_{0}^{t} (t-s)^{q-1} \sup_{0\leq\eta\leq s}E\big[\vert z^{N}_{m}(\eta)\vert^{p}\big]ds\bigg\rbrace
	\end{align*}
	which with the application of weakly singular Gronwall’s inequality (\cite{5}, Corollary 2) yields
	\begin{align*}
		E\big[\vert z^{N}_{m}(t)\vert^{p}\big]\leq C, \quad \forall t\in [0,T].
	\end{align*}
	Letting $m\to +\infty$ and using Fatou’s lemma to indicate
	\begin{align*}
		E\big[\vert z^{N}(t)\vert^{p}\big]\leq C, \quad \forall t\in [0,T].
	\end{align*}
	Additionally, by using the same logic and approach as in the previous proof for the scenario where $p \geq 2$, we can also derive the same results when $p = 2$. However, instead of utilizing Hölder’s inequality, we will substitute it with Cauchy-Schwarz’s inequality.
\end{proof}
\begin{lemma}\label{l.4}
	Let $0\leq t^{\prime} < t \leq T$, and $0<\alpha\leq\alpha_{i}\leq1$. If the Assumptions (1) and (4) satisfies, then there is a positive constant $C$ that is positive and not dependent on N. Such that for any integer $p\geq 2$,
	\begin{align*}
		E\Big[\vert \hat{z}^{N}(t)-\hat{z}^{N}(t^{\prime})\vert^{p}\Big]\leq C(t-t^{\prime})^{\alpha p}.
	\end{align*}
\end{lemma}
\begin{proof}
	The relation  (\ref{10}) implies that
	\begin{align*}
		E\big[\vert z^{N}(t)-z^{N}(t^{\prime})\vert^{p}\big]\leq& (n+3)^{p-1}\bigg\lbrace E\bigg[\sum_{i=1}^{n}\bigg\vert \int_{0}^{t}F_{i}(t,s,\hat{z}^{N}(s))ds-\int_{0}^{t^{\prime}}F_{i}(t,s,\hat{z}^{N}(s))ds\bigg\vert^{p}\\
		+&E\bigg[\bigg\vert \int_{0}^{t}G_{0}(t,s,\hat{z}^{N}(s))ds-\int_{0}^{t^{\prime}}G_{0}(t,s,\hat{z}^{N}(s))ds\bigg\vert^{p}\bigg]\\
		+&E\bigg[\bigg\vert \int_{0}^{t}G_{1}(t,s,\hat{z}^{N}(s))ds-\int_{0}^{t^{\prime}}G_{1}(t,s,\hat{z}^{N}(s))ds\bigg\vert^{p}\bigg]\\
		+&E\bigg[\bigg\vert \int_{0}^{t}G_{2}(t,s,\hat{z}^{N}(s))dW(s)-\int_{0}^{t^{\prime}}G_{2}(t,s,\hat{z}^{N}(s))dW(s)\bigg\vert^{p}\bigg]\bigg\rbrace\\
		:=&(n+3)^{p-1}\bigg(K_{0}+K_{1}+K_{2}+K_{3}\bigg)
	\end{align*}
	Applying Holder's inequality while considering assumptions (1) and (4), along with \cite{6} (from Lemma 3.1) and the previous statement of Lemma 3, we can demonstrate the following.
	\begin{align*}
		K_{0}=&E\bigg[\sum_{i=1}^{n}\bigg\vert \int_{0}^{t}F_{i}(t,s,\hat{z}^{N}(s))ds-\int_{0}^{t^{\prime}}F_{i}(t,s,\hat{z}^{N}(s))ds\bigg\vert^{p}\\
		\leq&  2^{p-1}\bigg\lbrace \sum_{i=1}^{n}E\bigg[\bigg\vert \int_{0}^{t^{\prime}}\big((t-s)^{\alpha_{i}-1}-(t^{\prime}-s)^{\alpha_{i}-1}\big)f_{i}(s,\hat{z}^{N}(s))ds\bigg\vert^{p}\bigg] \\
		+&\sum_{i=1}^{n}E\bigg[\bigg\vert \int_{t^{\prime}}^{t}(t-s)^{\alpha_{i}-1}f_{i}(s,\hat{z}^{N}(s))ds\bigg\vert^{p}\bigg] \\
		\leq& C\bigg\lbrace \sum_{i=1}^{n} \bigg(\int_{0}^{t^{\prime}}\big\vert(t-s)^{\alpha_{i}-1}-(t^{\prime}-s)^{\alpha_{i}-1}\big\vert ds\bigg)^{p-1}\\
		&\int_{0}^{t^{\prime}}\big\vert(t-s)^{\alpha_{i}-1}-(t^{\prime}-s)^{\alpha_{i}-1}\big\vert \bigg(1+E\big[\vert \hat{z}^{N}(s)\vert^{p}\big]\bigg)ds\\
		+&\sum_{i=1}^{n}\bigg(\int_{t^{\prime}}^{t}(t-s)^{\alpha_{i}-1}ds\bigg)^{p-1}\int_{t^{\prime}}^{t}(t-s)^{\alpha_{i}-1}\bigg(1+E\big[\vert \hat{z}^{N}(s)\vert^{p}\big]\bigg)ds\bigg\rbrace\\
		\leq& C\sum_{i=1}^{n} (t-t^{\prime})^{\alpha_{i}p}\leq  C (t-t^{\prime})^{\alpha p}
	\end{align*}
	By using the same technique, we will achieve the following results for $K_1$ and $K_2$.
	\begin{align*}
		K_{1}\leq C (t-t^{\prime})^{\alpha p}\quad and\quad
		K_{2}\leq C(t-t^{\prime})^{(1\land (1+\alpha-\beta_{1}) )p}
	\end{align*}
	Using Holder inequality, Burkholder-Davis-Gundy inequality, and Assumptions (1),(4), as well as Lemma 3. and, from \cite{6} (see lemma 3.2),
	\begin{align*}
		K_{3}\leq& C\bigg\lbrace E\bigg[\Big\vert \int_{0}^{t^{\prime}} ((t-s)^{\alpha-\beta_{2}}-(t^{\prime}-s)^{\alpha-\beta_{2}})^{2}\sup_{s\leq v\leq t}\big\vert g_{2}(v,s,\hat{z}^{N}(s))\big\vert^{2} ds\Big\vert^{p/2} \bigg]\\
		+&E\bigg[\Big\vert \int_{0}^{t^{\prime}} ((t^{\prime}-s)^{2(\alpha-\beta_{2})}\sup_{0\leq v\leq 1}\big\vert g_{2}((t-s)v+s,s,\hat{z}^{N}(s))-g_{2}((t^{\prime}-s)v+s,s,\hat{z}^{N}(s))\big\vert^{2} ds\Big\vert^{p/2} \bigg]\\
		+& E\bigg[\Big\vert \int_{t^{\prime}}^{t} ((t-s)^{2(\alpha-\beta_{2})}\sup_{s\leq v\leq t}\big\vert g_{2}(v,s,\hat{z}^{N}(s))\big\vert^{2} ds\Big\vert^{p/2} \bigg\rbrace\\
		\leq& C\bigg\lbrace E\bigg[\Big\vert \int_{0}^{t^{\prime}} ((t-s)^{\alpha-\beta_{2}}-(t^{\prime}-s)^{\alpha-\beta_{2}})^{2}ds\bigg\vert ^{\frac{p-2}{2}}\\
		&\int_{0}^{t^{\prime}}((t-s)^{\alpha-\beta_{2}}-(t^{\prime}-s)^{\alpha-\beta_{2}})^{2}\Big(1+E\big[\vert \hat{z}^{N}(s)\vert^{p}\big]\Big)ds\\
		+&(t-t^{\prime})^{2}\bigg\vert \int_{0}^{t^{\prime}} (t^{\prime}-s)^{2(\alpha-\beta_{2})}ds\bigg\vert^{\frac{p-2}{2}}\int_{0}^{t^{\prime}}(t^{\prime}-s)^{2(\alpha-\beta_{2})}\Big(1+E\big[\vert \hat{z}^{N}(s)\vert^{p}\big]\Big)ds\\
		+&\bigg\vert \int_{t^{\prime}}^{t} (t^{\prime}-s)^{2(\alpha-\beta_{2})}ds\bigg\vert^{\frac{p-2}{2}}\int_{0}^{t^{\prime}}(t^{\prime}-s)^{2(\alpha-\beta_{2})}\Big(1+E\big[\vert \hat{z}^{N}(s)\vert^{p}\big]\Big)ds\\
		\leq& 	\begin{cases}
			C(t-t^{\prime})^{(2-\varepsilon)\frac{p}{2}},\quad \quad \quad \quad \quad if\quad \alpha-\beta_{2}=1/2,\\
			\\
			C(t-t^{\prime})^{2\land (1+2(\alpha-\beta_{2}))\frac{p}{2}},\quad otherwises.
		\end{cases}
	\end{align*}
	If $\alpha - \beta_2 = \frac{1}{2}$, then $2-2\alpha \in (0,1)$ because $\beta_2 \in (0,\frac{1}{2})$. By setting $\varepsilon = 2-2\alpha$, we have 
	\begin{align*}
		K_{3}\leq C (t-t^{\prime})^{\alpha p}.
	\end{align*}
	As a result, applying these results on the right side of the first inequality above, we will acquire
	\begin{align*}
		E\Big[\vert \hat{z}^{N}(t)-\hat{z}^{N}(t^{\prime})\vert^{p}\Big]\leq(n+3)^{p}(K_{0}+K_{1}+K_{2}+K_{3})\leq C(t-t^{\prime})^{\alpha p}.
	\end{align*}
\end{proof}
\subsection{Proof of the Theorem 1.}
\begin{proof}
	To keep things simple, we just prove the global Lipschitz condition. Indeed, according to the proof under the global Lipschitz condition, using the truncation functions
	\begin{align*}
		G^{n}_{i}(t,s,z):=\begin{cases}
			G_{i}(t,s,z),\quad if\quad \vert z\vert\leq n,\\
			\quad 	\quad	\quad	\quad	\quad	\quad	\quad 	\quad	\quad 	\quad	\quad 	\quad	\quad i=0,1,2.\\
			G_{i}(t,s,\frac{nz}{\vert z\vert}),\quad if\quad \vert z\vert> n.
		\end{cases}
	\end{align*}
	Using a similar proof technique to that of Theorem 3.4 in Chapter 2 of \cite{7}, we can readily deduce the desired outcome for the situation where the local Lipschitz condition holds for any positive integer $n\geq1$.

	\textbf{Uniqueness.}  Let $z(t)$ and $\tilde{z}(t)$ be two solutions of the Stochastic Fractional Neutral Integro-differential Equation  (\ref{A}) on the same probability space with $z(0)=\tilde{z}(0)$.  Theorem \ref{t.5}  shows that $z(t)$ and $\tilde(t)$ are also two solutions to the Stochastic Volterra Integral Equation (\ref{99}). Using Hölder's inequality, Ito isometry, and the Lipschitz condition (\ref{4}), one can deduce from (\ref{789}) that
	\begin{align}
		E\Big[\big\vert z(t)-\tilde{z}(t)\big\vert^{2}\Big]\leq& \tilde{C} E\bigg\lbrace \sum_{i=0}^{n} \bigg\vert \int_{0}^{t} (t-s)^{\alpha_{i}-1} \big(f_{i}(s,z(s))-f_{i}(s,\tilde{z}(s))\big)ds\bigg\vert^{2}\nonumber\\
		+&\bigg\vert \int_{0}^{t} (t-s)^{\alpha-1} \big(g_{0}(s,z(s))-g_{0}(s,\tilde{z}(s))\big)ds\bigg\vert^{2}\nonumber\\
		+&\bigg\vert \int_{0}^{t} (t-s)^{\alpha-\beta_{1}} \sup_{s\leq v\leq t}\big\vert g_{1}(v,s,z(s))-g_{1}(v,s,\tilde{z}(s))\big\vert ds\bigg\vert^{2}\nonumber\\
		+& \int_{0}^{t} (t-s)^{2(\alpha-\beta_{2})} \sup_{s\leq v\leq t}\big\vert g_{2}(v,s,z(s))-g_{2}(v,s,\tilde{z}(s))\big\vert^{2} ds\bigg\rbrace\nonumber\\
		\leq&\tilde{C}\Bigg\lbrace \sum_{i=0}^{n} \int_{0}^{t}(t-s)^{\alpha_{i}-1}E\Big[\big\vert z(s)-\tilde{z}(s)\big\vert^{2}\Big]ds+\int_{0}^{t}(t-s)^{\alpha-1}E\Big[\big\vert z(s)-\tilde{z}(s)\big\vert^{2}\Big]ds\Bigg\rbrace		
	\end{align}
Assume that $0<\alpha<\alpha_{1}<\dots<\alpha_{n}\leq 1$. Then we have 
\begin{align}\label{11}
		E\Big[\big\vert z(t)-\tilde{z}(t)\big\vert^{2}\Big]\leq&\tilde{C}\Bigg\lbrace \sum_{i=0}^{n} \int_{0}^{t}(t-s)^{\alpha_{i}-1}E\Big[\big\vert z(s)-\tilde{z}(s)\big\vert^{2}\Big]ds+\int_{0}^{t}(t-s)^{\alpha-1}E\Big[\big\vert z(s)-\tilde{z}(s)\big\vert^{2}\Big]ds\Bigg\rbrace\nonumber\\
		\leq & \tilde{C}\Bigg\lbrace \sum_{i=0}^{n} \int_{0}^{t}(t-s)^{\alpha_{i}-\alpha+\alpha-1}E\Big[\big\vert z(s)-\tilde{z}(s)\big\vert^{2}\Big]ds+\int_{0}^{t}(t-s)^{\alpha-1}E\Big[\big\vert z(s)-\tilde{z}(s)\big\vert^{2}\Big]ds\Bigg\rbrace\nonumber\\
		\leq& \tilde{C}\Bigg\lbrace n\max\{1, T^{\alpha_{n}-\alpha}\} \int_{0}^{t}(t-s)^{\alpha-1}E\Big[\big\vert z(s)-\tilde{z}(s)\big\vert^{2}\Big]ds+\int_{0}^{t}(t-s)^{\alpha-1}E\Big[\big\vert z(s)-\tilde{z}(s)\big\vert^{2}\Big]ds\Bigg\rbrace\nonumber\\
		\leq& C\int_{0}^{t}(t-s)^{\alpha-1}E\Big[\big\vert z(s)-\tilde{z}(s)\big\vert^{2}\Big]ds
\end{align}
	Then, weakly singular  Gronwall's inequality (\cite{5}, Corollary 2), we will obtain.  
	\begin{align*}
		E\Big[\vert z(t)-\tilde{z}(t)\vert^{2}\Big]=0, \quad \forall t\in[0,T].
	\end{align*}
	which indicates
	\begin{align*}
		P\big\lbrace \vert z(t)-\tilde{z}(t)\vert=0,\quad \forall t\in Q\cap[0,T]\big\rbrace=1,
	\end{align*}
	where $Q$ shows the set of all rational numbers. It follows from the continuity of $\vert z(t)-\tilde{z}(t)\vert$ with respect to $t$ that 
	\begin{align*}
		P\big\lbrace \vert z(t)-\tilde{z}(t)\vert=0,\quad \forall t\in [0,T]\big\rbrace=1,
	\end{align*}
	The uniqueness has been proven.
	
	\textbf{Existence.}  Let $M \geq N \geq 1$. Similar to the derivation of the estimate (\ref{11}), one can claim from  (\ref{10}) that for any $p \geq 2$,
	
	\begin{align*}
		E\Big[\vert z^{M}(t)-z^{N}(t) \vert^{p}\Big]\leq&\bar{C}\bigg\lbrace \sum_{i=0}^{n} \int_{0}^{t}(t-s)^{\alpha_{i}-1}E\Big[\big\vert \hat{z}^{M}(s)-\hat{z}^{N}(s)\big\vert^{p}\Big]ds\\
		+&\int_{0}^{t}(t-s)^{\alpha-1}E\Big[\big\vert\hat{z}^{M}(s)-\hat{z}^{N}(s)\big\vert^{p}\Big]ds\Bigg\rbrace\\
			&\leq\bar{C}\bigg\lbrace n\max\{1, T^{\alpha_{n}-\alpha}\} \int_{0}^{t}(t-s)^{\alpha-1}E\Big[\big\vert \hat{z}^{M}(s)-\hat{z}^{N}(s)\big\vert^{p}\Big]ds\\
			+&\int_{0}^{t}(t-s)^{\alpha-1}E\Big[\big\vert\hat{z}^{M}(s)-\hat{z}^{N}(s)\big\vert^{p}\Big]ds\Bigg\rbrace\\
		\leq& C\int_{0}^{t}(t-s)^{\alpha-1}E\Big[\big\vert\hat{z}^{M}(s)-\hat{z}^{N}(s)\big\vert^{p}\Big]ds
	\end{align*}where  the positive constant $C$, which is not dependent on $M$ and $N$. Following that, using the triangle inequality with Lemma 4 and the weakly singular Gronwall's inequality, one can derive that $z^{N}(t)$ is a Cauchy sequence with a limit $z(t)$ in $L^{p}(\Omega, R^{d})$. Clearly, $z(t)$ is $\mathcal{F}_{t}$-adapted. Furthermore, Lemma 4 and Fatou's lemma imply that
	\begin{align*}
		E\Big[\vert \hat{z}^{N}(t)-\hat{z}^{N}(t^{\prime})\vert^{p}\Big]\leq C(t-t^{\prime})^{\alpha p}.
	\end{align*}
	The process $z(t)$ has a continuous version according to Kolmogorov's continuity condition. Allowing $N \to+\infty$ in (\ref{10}) implies that the continuous version solves the SVIE (\ref{789}) on $[0,T]$. The continuous variant is likewise a solution to SFNIDE  (\ref{A}), according to  Theorem \ref{t.5}. Finally, according to Lemma 3, letting $N \to +\infty$ produces the conclusion.
	\begin{align*}
		E\Big[\vert \hat{z}^{N}(t)\vert^{p}\Big]\leq+\infty,\quad \forall t\in [0,T].
	\end{align*}
	The proof is complete.
\end{proof}

\subsection{Continious dependence of solutions on the initial value.}\label{sec4}
We are going to examine the continuous dependence of solutions on the initial value in this section of the article.
\begin{definition}\label{d.2}
	The solution of the SFNIDE (\ref{A}) is said to be continuous depending on the value in a mean square sense if there is a positive number $\eta$ for every $\varepsilon$ such that 
	\begin{align*}
		E\big[\vert z(t)-\bar{z}(t)\vert^{2}\big]<\varepsilon,\quad \forall\in [0,T],
	\end{align*}
	provided that $	E\big[\vert z_{0}(t)-\bar{z}_{0}(t)\vert^{2}\big]<\eta$, where $\bar{z}(t)$ is each other solution to  (\ref{A}) with the initial value $z_{0}\in R^{d}$.
\end{definition}
\subsubsection{Proof of the Theorem 2.}
\begin{proof}
	Let $z(t)$ and $\bar{z}(t)$ be two solutions to equation  (\ref{A}) with the  distinct initial values $z_{0}$ and $\bar{z}_{0}$, accordingly. Then, $z(t)$ and $\bar{z}(t)$ are also two solutions to equation (4). To make it easier, let $e(t)=z(t)-\bar{z}(t)$,
	\begin{align}\label{13}
		\rho_{k}=\inf \lbrace t\geq 0 : \vert z(t)\vert\geq k\rbrace, \quad \mu_{k}=\inf \lbrace t\geq 0 : \vert \bar{z}(t)\vert\geq k\rbrace,\quad \sigma_{k}=\rho_{k}\land \mu_{k}
	\end{align}
	for every integer $k\geq 1$. 
	Then, by using Young inequality (\ref{yon}) for each $\delta>0$, we have
	\begin{align}
		&E\big[\vert e(t)\vert^{2}\big]=E\big[\vert e(t)\vert^{2} 1_{\lbrace \rho_{k}>T, \mu_{k}>T\rbrace}\big]+	E\big[\vert e(t)\vert^{2}1_{\lbrace \rho_{k}\leq T,\quad or\quad \mu_{k}\leq T\rbrace}\big]\nonumber\\
		\quad 	\quad	\quad 	\quad	\quad\\
		&\leq E\big[\vert e(t\land \sigma_{k})\vert^{2}1_{\lbrace \sigma_{k}>T\rbrace}\big]+\frac{2\delta}{p} E\big[\vert e(t)\vert^{p}\big]+\frac{p-2}{p\delta^{\frac{2}{p-2}}}P(\rho_{k}\leq T\quad or\quad \mu_{k}\leq T)\nonumber
	\end{align}
	On the one hand, Theorem \ref{t.1} provides  
	\begin{align*}
		E\big[\vert e(t)\vert^{p}\big]\leq 2^{p-1}E\big[\vert z(t)\vert^{p}+\vert \bar{z}(t)\vert^{p}\big]\leq 2^{p}M,
	\end{align*}
	in this proof and throughout the rest of it, the variable $M$ represents a positive constant that does not depend on the values of $k$ and $\delta$. On the other hand, we will obtain 
	\begin{align}\label{15}
		&P(\rho_{k}\leq T\quad or\quad \mu_{k}\leq T)\leq P(\rho_{k}\leq T)+P( \mu_{k}\leq T)\nonumber\\
		&=E\bigg(1_{\lbrace \rho_{k}\leq T\rbrace}\frac{\vert z(\rho_{k})\vert^{p}}{k^{p}}\bigg)+E\bigg(1_{\lbrace \mu_{k}\leq T\rbrace}\frac{\vert z(\mu_{k})\vert^{p}}{k^{p}}\bigg)\\
		&\leq \frac{1}{k^{p}}\Big\lbrace E\big(\vert z(\rho_{k}\land T)\vert^{p}\big)+E\big(\vert z(\mu_{k}\land T)\vert^{p}\big)\Big\rbrace\leq \frac{2M}{k^{p}}\nonumber
	\end{align}
Thus, the inequality (\ref{15}) gives 
\begin{align}\label{16}
	E\big[\vert e(t)\vert^{2}\big]\leq E\big[\vert e(t\land \sigma_{k}\vert^{2})\big]+\frac{2^{p+1}\delta M}{p}+\frac{2(p-2)M}{p \delta^{\frac{2}{p-2}}k^{p}}.
\end{align}
By using similar technique with (\ref{11}), we have.
	\begin{align}\label{xt}
		&E\big[\vert e(t\land \sigma_{k}\vert^{2})\big]=E\big[\vert z(t\land \sigma_{k})-\bar{z}(t\land \sigma_{k})\vert^{2}\big]\nonumber\\
		\leq& C E\bigg\lbrace \vert z_{0}-\bar{z}_{0}\vert^{2}+\sum_{i=1}^{n}\bigg\vert \int_{0}^{t\land \sigma_{k}} (t\land \sigma_{k}-s)^{\alpha_{i}-1}\big(f_{i}(s,z(s))-f_{i}(s,\bar{z}(s))\big)ds\bigg\vert^{2}\nonumber\\
		+&\bigg\vert \int_{0}^{t\land \sigma_{k}} (t\land \sigma_{k}-s)^{\alpha-1}\big( g_{0}(s,z(s))-g_{0}(s,\bar{z}(s))\big) ds\bigg\vert^{2}\nonumber\\
		+&\bigg\vert \int_{0}^{t\land \sigma_{k}} (t\land \sigma_{k}-s)^{\alpha-\beta_{1}}\sup_{s\leq v\leq t}\big\vert g_{1}(v,s,z(s))-g_{1}(v,s,\bar{z}(s))\big\vert ds\bigg\vert^{2}\nonumber\\
		+& \int_{0}^{t\land \sigma_{k}} (t\land \sigma_{k}-s)^{2(\alpha-\beta_{2})}\sup_{s\leq v\leq t}\big\vert g_{2}(v,s,z(s))-g_{2}(v,s,\bar{z}(s))\big\vert^{2} ds\bigg\rbrace\nonumber\\
		\leq&C\bigg\lbrace E\big[\vert z_{0}-\bar{z}_{0}\vert^{2}\big]+\sum_{i=1}^{n} \int_{0}^{t\land \sigma_{k}}  (t\land \sigma_{k}-s)^{\alpha_{i}-1}E\big[\vert e(s\land \sigma_{k})\vert^{2}\big]ds\nonumber\\
		+&\int_{0}^{t\land \sigma_{k}}  (t\land \sigma_{k}-s)^{\alpha-1}E\big[\vert e(s\land \sigma_{k})\vert^{2}\big]ds\bigg\rbrace
	\end{align} 
Suppose that $0<\alpha\leq \alpha_{1}<\alpha_{2}<\dots,<\alpha_{n}\leq1$. We will obtain the following relation.
\begin{align*}
	&\sum_{i=1}^{n} \int_{0}^{t\land \sigma_{k}}  (t\land \sigma_{k}-s)^{\alpha_{i}-1}E\big[\vert e(s\land \sigma_{k})\vert^{2}\big]ds+\int_{0}^{t\land \sigma_{k}}  (t\land \sigma_{k}-s)^{\alpha-1}E\big[\vert e(s\land \sigma_{k})\vert^{2}\big]ds\\
	=&\sum_{i=1}^{n} \int_{0}^{t\land \sigma_{k}}  (t\land \sigma_{k}-s)^{\alpha_{i}-\alpha+\alpha-1}E\big[\vert e(s\land \sigma_{k})\vert^{2}\big]ds+\int_{0}^{t\land \sigma_{k}}  (t\land \sigma_{k}-s)^{\alpha-1}E\big[\vert e(s\land \sigma_{k})\vert^{2}\big]ds\\
	\leq & n \max\{1, T^{\alpha_{n}-\alpha}\}\int_{0}^{t\land \sigma_{k}}  (t\land \sigma_{k}-s)^{\alpha-1}E\big[\vert e(s\land \sigma_{k})\vert^{2}\big]ds+\int_{0}^{t\land \sigma_{k}}  (t\land \sigma_{k}-s)^{\alpha-1}E\big[\vert e(s\land \sigma_{k})\vert^{2}\big]ds\\
	\leq &\bar{C} \int_{0}^{t\land \sigma_{k}}  (t\land \sigma_{k}-s)^{\alpha-1}E\big[\vert e(s\land \sigma_{k})\vert^{2}\big]ds\\
\end{align*}
As a result, by using the above last expression in the \eqref{xt}, we achieve

\begin{align}
		&E\big[\vert e(t\land \sigma_{k}\vert^{2})\big]\leq K \bigg\lbrace E\big[\vert z_{0}-\bar{z}_{0}\vert^{2}\big]+\int_{0}^{t\land \sigma_{k}}  (t\land \sigma_{k}-s)^{\alpha-1}E\big[\vert e(s\land \sigma_{k})\vert^{2}\big]ds\bigg\rbrace
\end{align}
Where $K=\max\{C, \bar{C}\}$.
Applying the weakly singular Gronwall inequality (\cite{5},Corollary 2) and subsequently substituting the obtained outcome into equation (\ref{16}) results in
\begin{align*}
		E\big[\vert e(t)\vert^{2}\big]\leq C_{k}E\big[\vert z_{0}-\bar{z}_{0}\vert^{2}\big]+\frac{2^{p+1}\delta M}{p}+\frac{2(p-2)M}{p \delta^{\frac{2}{p-2}}k^{p}},
\end{align*}
where the positive constant $C_{k}$ depends on $k$, but not on $\delta$. It is observed that $\delta$ and $k$ can be can chosen such that
\begin{align*}
	\frac{2^{p+1}\delta M}{p}<\frac{\varepsilon}{3},\quad \frac{2(p-2)M}{p \delta^{\frac{2}{p-2}}k^{p}} <\frac{\varepsilon}{3},
\end{align*}
and there is a positive number $\eta>0$ such that $C_{k}E\big[\vert z_{0}-\bar{z}_{0}\vert^{2}\big]<C_{k}\eta<\frac{\varepsilon}{3}$
Consequently,
\begin{align*}
	E\big[\vert e(t)\vert^{2}\big]<\varepsilon,\quad \forall\in [0,T],
\end{align*}
The proof is complete.
\end{proof}
\section{Strong convergence and convergence rate of the Euler-Maruyama method.}

	It becomes particularly crucial to take into account effective numerical techniques because it isn't easy to acquire the closed-form solution to the SFNIDE  (\ref{A}). Even though the EM approximation (\ref{10}) offers a numerical approximation in the preceding section, it will require many stochastic integral calculations. In this section, we now change approximation (\ref{10}) to reduce this cost.
	
	Using the EM approximation (\ref{10}) and the same values, we can change it utilizing the left rectangle rule (\cite{3},\cite{11}) as
	
	\begin{align}\label{17}
		Z(t)=&z_{0}+\sum_{i=1}^{n}\int_{0}^{t}F_{i}(t,\hat{s},\hat{Z}(s))ds+\int_{0}^{t}G_{0}(t,\hat{s},\hat{Z}(s))ds\nonumber\\
		+&\int_{0}^{t}G_{1}(t,\hat{s},\hat{Z}(s))ds+\int_{0}^{t}G_{2}(t,\hat{s},\hat{Z}(s))dW(s),
	\end{align} 
where $\hat{s}=t_{n}$ for $s\in [t_{n},t_{n+1})$ and the simple step process $\hat{Z}(t)=\sum_{n=0}^{N}Z(t_{n})\dot 1_{ [t_{n},t_{n+1})}(t)$. We use the discrete-time form of the algorithm in our implementation
\begin{align}\label{18}
	Z_{n}=&Z(t_{n})=z_{0}+\sum_{j=0}^{n-1}\sum_{i=1}^{n} F_{i}(t_{n},\hat{s},Z_{j})h+\sum_{j=0}^{n-1} G_{0}(t_{n},\hat{s},Z_{j})h\nonumber\\
	+&\sum_{j=0}^{n-1} G_{1}(t_{n},\hat{s},Z_{j})h+\sum_{j=0}^{n-1} G_{2}(t_{n},\hat{s},Z_{j})\Delta W_{j},\quad n=1,\dots,N.
\end{align}
with $Z_{0}=z_{0}$, where $\Delta W_{j}=W(t_{j+1})-W(t_{j}),\quad j=0,1,\dots, N-1$ denote the increments of Brownian motion $W(t)$. We will significantly minimize the number of calculations because we only must simulate these increments and do not need to compute other stochastic integrals.

\subsection{Strong convergence.}

\begin{lemma}\label{l.5}
	If Assumptions (4) satisfies, then there is a constant $C>0$ independent of $h$ such that for each integer $p\geq 2$,
	\begin{align*}
		E\big[\vert Z(t)\vert^{p}\big]\leq C\quad and\quad 	E\big[\vert \hat{Z}(t)\vert^{p}\big]\leq C,\quad \forall t\in [0,T]
	\end{align*}
\end{lemma}
\begin{proof}
	The proof has been omitted since it is exactly the same as the proof of Lemma 3.
\end{proof}
\begin{lemma}\label{l.6}
	If the Assumption (4) satisfies, then there is a constant $C>0$ independent of $h$ such that for each integer $p\geq2$,
	\begin{align*}
		E\big[\vert Z(t)-\hat{Z}(t)\vert^{2}\big]\leq Ch^{2\alpha},\quad \forall t\in [0,T].
	\end{align*}
\end{lemma}

\begin{proof} For every $t\in [0,T]$, there exists a unique integer $n$ such that $t\in [t_{n},t_{n+1})$ and $\hat{Z}(t)=Z(t_{n})$. Then, it follows from (\ref{17}) that
\begin{align*}
&	E\big[\vert Z(t)-\hat{Z}(t)\vert^{2}\big]=	E\big[\vert Z(t)-Z(t_{n})\vert^{2}\big]\\
	\leq&(n+3)\bigg\lbrace \sum_{i=1}^{n}E\Big[ \Big\vert \int_{0}^{t}F_{i}(t,\hat{s},\hat{Z}(s))ds-\int_{0}^{t_{n}}F_{i}(t_{n},\hat{s},\hat{Z}(s))ds\Big\vert^{2}\Big]\\
	+&E\Big[ \Big\vert \int_{0}^{t}G_{0}(t,\hat{s},\hat{Z}(s))ds-\int_{0}^{t_{n}}G_{0}(t_{n},\hat{s},\hat{Z}(s))ds\Big\vert^{2}\Big]\\
	+&E\Big[ \Big\vert \int_{0}^{t}G_{1}(t,\hat{s},\hat{Z}(s))ds-\int_{0}^{t_{n}}G_{1}(t_{n},\hat{s},\hat{Z}(s))ds\Big\vert^{2}\Big]\\
	+&E\Big[ \Big\vert \int_{0}^{t}G_{2}(t,\hat{s},\hat{Z}(s))dW(s)-\int_{0}^{t_{n}}G_{2}(t_{n},\hat{s},\hat{Z}(s))dW(s)\Big\vert^{2}\Big]\bigg\rbrace\\
	=&(n+3)\big\lbrace\hat{K}_{0}+\hat{K}_{1}+\hat{K}_{2}+\hat{K}_{3}\big\rbrace.
\end{align*}	
Imposing Hölder’s inequality and Assumptions 1, 4 as well as  Lemmas 1 and 5 to get
\begin{align*}
	\hat{K}_{0}\leq&2\sum_{i=1}^{n}\bigg\lbrace E\Big[\Big\vert \int_{0}^{t_{n}} \big((t-\hat{s})^{\alpha_{i}-1}-(t_{n}-\hat{s})^{\alpha_{i}-1}\big)f_{i}(\hat{s},\hat{Z}(s))ds\Big\vert^{2}\Big]\\
	+&E\Big[\Big\vert \int_{t_{n}}^{t}(t-\hat{s})^{\alpha_{i}-1}f_{i}(\hat{s},\hat{Z}(s))ds\Big\vert^{2}\Big]\\
	\leq& C\sum_{i=1}^{n} \bigg\lbrace \int_{0}^{t_{n}} \vert (t-\hat{s})^{\alpha_{i}-1}-(t_{n}-\hat{s})^{\alpha_{i}-1}\vert ds\\
	&\int_{0}^{t_{n}} \vert (t-\hat{s})^{\alpha_{i}-1}-(t_{n}-\hat{s})^{\alpha_{i}-1}\vert\big(1+E\big[\vert \hat{Z}(s)\vert^{2}\big]\big) ds\\
	+&\int_{t_{n}}^{t}(t-s)^{\alpha_{i}-1}ds\int_{t_{n}}^{t}(t-s)^{\alpha_{i}-1}\big(1+E\big[\vert \hat{Z}(s)\vert^{2}\big]\big) ds\bigg\rbrace\\
	\leq&C\sum_{i=1}^{n}h^{2\alpha_{i}}\leq Ch^{2\alpha}.
\end{align*}
By using the similar method, we can write the followings for $\hat{K}_{1}$ and $\hat{K}_{2}$.
\begin{align*}
	\hat{K}_{1}\leq Ch^{2\alpha} \quad and\quad \hat{K}_{2}\leq C h^{2\land 2(1+\alpha-\beta_{1})}\leq Ch^{2\alpha}.
\end{align*}
Following that, we analyze the case of the expression $\hat{K}_{3}$. We can achieve the following result for $\hat{K}_{3}$ using the same idea as $K_{3}$ in lemma 2. 
\begin{align*}
	\hat{K}_{3}
\leq	\begin{cases}
	Ch^{(2-\varepsilon)},\quad \quad \quad \quad \quad if\quad \alpha-\beta_{2}=\frac{1}{2},\\
	\\
	Ch^{2\land (1+2(\alpha-\beta_{2}))},\quad otherwises.
\end{cases}
\end{align*}
If $\alpha - \beta_2 = \frac{1}{2}$, then $2-2\alpha \in (0,1)$ because $\beta_2 \in (0,\frac{1}{2})$. By setting $\varepsilon = 2-2\alpha$, we have 
\begin{align*}
\hat{K}_{3}\leq C h^{2\alpha}.
\end{align*}
If $\alpha - \beta_2 \not= \frac{1}{2}$, then the assumption $\beta_2 \in (0,\frac{1}{2})$implies
\begin{align*}
	\hat{K}_{3}\leq Ch^{2\land (1+2\alpha-2\beta_{2})}\leq C h^{2\alpha}.
\end{align*}
Then, we have 
\begin{align*}
	E\big[\vert Z(t)-\hat{Z}(t)\vert^{2}\big]\leq (n+3)\big(\hat{K}_{0}+\hat{K}_{1}+\hat{K}_{2}+\hat{K}_{3}\big)\leq Ch^{2\alpha}.
\end{align*}
\end{proof}
Below, we prove the Euler-Maruyama method's mean-square convergence theorem (\ref{17}).
\begin{theorem}\label{t.7}
The Euler-Maruyama approach $Z(t)$ stated in $(\ref{17})$ converges to the exact solution $z(t)$ in mean square sense under assumptions $1-4$, i.e.,
\begin{align*}
	\lim_{h\to 0}E\Big[\big\vert Z(t)-z(t)\big\vert^{2}\Big]=0,\quad \forall t\in [0,T].
\end{align*}
\end{theorem}
\begin{proof}
	For simplicity, let the error $e(t) = Z(t)-z(t)$ and for every integer $m \geq 1$,
	\begin{align*}
		\theta_{m}=\inf \lbrace t\geq 0: \vert Z(t)\vert \geq m\rbrace,\quad \gamma_{m}=\theta_{m}\land \rho_{m},
	\end{align*}
where $\rho_{m}$ has been given in (\ref{13}). Resembling to (\ref{16}), it is follows from Lemma 5 that for each $\delta>0$,
\begin{align}\label{19}
	E\big[\vert e(t)\vert^{2}\big]\leq E\big[\vert e(t\land \gamma_{m})\vert^{2}\big]+\frac{2^{p+1}\delta M}{p}+\frac{2(p-2)M}{p \delta^{\frac{2}{p-2}}m^{p}},
\end{align}
Here and throughout the rest of the proof, $M$ refers to a positive constant that is independent of $m, \delta$, and $h$. We then concentrate on estimating the term on the right side of the equation (\ref{19}). Due to Hölder's inequality, formulas (\ref{4}) and (\ref{17}) imply that
	\begin{align}\label{20}
		E\Big[\vert e(t\land \gamma_{m})\vert^{2}\Big]\leq&(2n+6)E\bigg\lbrace \sum_{i=1}^{n}\bigg\vert \int_{0}^{t\land \gamma_{m}} \Big(F_{i}(t\land \gamma_{m}, \hat{s}, \hat{Z}(s))-F_{i}(t\land \gamma_{m}, s, \hat{Z}(s))\Big)ds\bigg\vert^{2}\nonumber\\
		+&\sum_{i=1}^{n}\bigg\vert \int_{0}^{t\land \gamma_{m}} \Big(F_{i}(t\land \gamma_{m}, s, \hat{Z}(s))-F_{i}(t\land \gamma_{m}, s, z(s))\Big)ds\bigg\vert^{2}\nonumber\\
		+&\bigg\vert \int_{0}^{t\land \gamma_{m}} \Big(G_{0}(t\land \gamma_{m}, \hat{s}, \hat{Z}(s))-G_{0}(t\land \gamma_{m}, s, \hat{Z}(s))\Big)ds\bigg\vert^{2}\nonumber\\
		+&\bigg\vert \int_{0}^{t\land \gamma_{m}}\Big( G_{0}(t\land \gamma_{m}, s, \hat{Z}(s))-G_{0}(t\land \gamma_{m}, s, z(s))\Big)ds\bigg\vert^{2}\nonumber\\
		+&\bigg\vert \int_{0}^{t\land \gamma_{m}} \Big(G_{1}(t\land \gamma_{m}, \hat{s}, \hat{Z}(s))-G_{1}(t\land \gamma_{m}, s, \hat{Z}(s))\Big)ds\bigg\vert^{2}\nonumber\\
		+&\bigg\vert \int_{0}^{t\land \gamma_{m}} \Big(G_{1}(t\land \gamma_{m}, s, \hat{Z}(s))-G_{1}(t\land \gamma_{m}, s, z(s))\Big)ds\bigg\vert^{2}\nonumber\\
		+&\bigg\vert \int_{0}^{t\land \gamma_{m}} \Big(G_{2}(t\land \gamma_{m}, \hat{s}, \hat{Z}(s))-G_{2}(t\land \gamma_{m}, s, \hat{Z}(s))\Big)dW(s)\bigg\vert^{2}\nonumber\\
		+&\bigg\vert \int_{0}^{t\land \gamma_{m}} \Big(G_{2}(t\land \gamma_{m}, s, \hat{Z}(s))-G_{2}(t\land \gamma_{m}, s, z(s))\Big)dW(s)\bigg\vert^{2}\bigg\rbrace\nonumber\\
		=:&(2n+6)\Big\lbrace H_{1}+H_{2}+H_{3}+H_{4}+H_{5}+H_{6}+H_{7}+H_{8}\Big\rbrace.
	\end{align}
Using Lemmas 1,2 and, 5,  Assumptions 1, 2, 4, and Cauchy–Schwarz's inequality, it can be inferred from the derivation steps of $\hat{K}_{0}$, $\hat{K}_{1}$, $\hat{K}_{2}$, and $\hat{K}_{3}$ (refer to Lemma 6's proof) that
\begin{align}\label{21}
	H_{1}+H_{3}+H_{5}+H_{7}\leq Ch^{2\alpha}.
\end{align}
The derivation steps in (\ref{11}) imply, via Cauchy–Schwarz’s inequality, Assumption 3, and Ito isometry, that
\begin{align}\label{22}
	H_{2}+H_{4}+H_{6}+H_{8}\leq&C^{\prime} K^{2}_{m} \bigg\lbrace \sum_{i=1}^{n} \int_{0}^{t\land \gamma_{m}}(t\land \gamma_{m}-s)^{\alpha_{i}-1}E\Big[\vert \hat{Z}(s)-Z(s)\vert^{2}+\vert e(s\land \gamma_{m})\vert^{2}\Big]ds\nonumber\\
	+&\int_{0}^{t\land \gamma_{m}}(t\land \gamma_{m}-s)^{\alpha-1}E\Big[\vert \hat{Z}(s)-Z(s)\vert^{2}+\vert e(s\land \gamma_{m})\vert^{2}\Big]ds\bigg\rbrace
\end{align}
Suppose that $0<\alpha<\alpha_{1}<\dots<\alpha_{n}\leq 1$. Then we will get
\begin{align}\label{22}
	H_{2}+H_{4}+H_{6}+H_{8}\leq&C^{\prime} K^{2}_{m} \bigg\lbrace n\max\{1, T^{\alpha_{n}-\alpha}\} \int_{0}^{t\land \gamma_{m}}(t\land \gamma_{m}-s)^{\alpha-1}E\Big[\vert \hat{Z}(s)-Z(s)\vert^{2}+\vert e(s\land \gamma_{m})\vert^{2}\Big]ds\nonumber\\
	+&\int_{0}^{t\land \gamma_{m}}(t\land \gamma_{m}-s)^{\alpha-1}E\Big[\vert \hat{Z}(s)-Z(s)\vert^{2}+\vert e(s\land \gamma_{m})\vert^{2}\Big]ds\bigg\rbrace\nonumber\\
	\leq &C K^{2}_{m} \int_{0}^{t\land \gamma_{m}}(t\land \gamma_{m}-s)^{\alpha-1}E\Big[\vert \hat{Z}(s)-Z(s)\vert^{2}+\vert e(s\land \gamma_{m})\vert^{2}\Big]ds
\end{align}
Now that we have combined (\cite{20}–\cite{22}), we can use Lemma 6 and weakly singular Gronwall's inequality to get at
\begin{align}\label{23}
	E\Big[\vert e(t\land \gamma_{m})\vert^{2}\Big]\leq C_{m}h^{2\alpha}
\end{align}
where the  $C_{m}>0$ is a constant depends on $m$, but not on $h$ and $\delta$. Inserting (\ref{23}) into (\ref{19}) arrive
\begin{align*}
		E\big[\vert e(t)\vert^{2}\big]\leq C_{m}h^{2\alpha}+\frac{2^{p+1}\delta M}{p}+\frac{2(p-2)M}{p \delta^{\frac{2}{p-2}}m^{p}},
\end{align*}
Thus, for every given $\varepsilon>0$, one can choose $\delta$ and $m$ such that 
\begin{align*}
	\frac{2^{p+1}\delta M}{p}<\frac{\varepsilon}{3},\quad \frac{2(p-2)M}{p \delta^{\frac{2}{p-2}}m^{p}} <\frac{\varepsilon}{3},
\end{align*}
and then $h$ can be taken sufficiently small such that $C_{m}h^{2\alpha}<\frac{\varepsilon}{3}$. As a consequence,
Consequently,
\begin{align*}
	\lim_{h\to 0}E\big[\vert e(t)\vert^{2}\big]0,\quad \forall\in [0,T],
\end{align*}
The proof is complete.
\end{proof}
\subsection{Convergence rate.}
We provide the following theorem since the numerical scheme's convergence rate can demonstrate its computing efficiency.
\begin{theorem}\label{t.8}
	Assuming that the global Lipschitz condition (2) and the assumptions of  Theorem \ref{t.7} are held, there is a positive constant $C$ independent of $h$ such that 
	\begin{align*}
		\bigg[E\big[\vert Z(t)-z(t)\vert^{2}\big]\bigg]^{\frac{1}{2}}\leq Ch^{\alpha},\quad \forall t\in [0,T].
	\end{align*}
\end{theorem}
\begin{proof}
	It follows from (\ref{789}) and (\ref{17}) that
	\begin{align}\label{24}
		E\Big[\vert Z(t)-z(t)\vert^{2}\Big]\leq&(2n+6)E\bigg\lbrace \sum_{i=1}^{n}\bigg\vert \int_{0}^{t} \Big(F_{i}(t, \hat{s}, \hat{Z}(s))-F_{i}(t, s, \hat{Z}(s))\Big)ds\bigg\vert^{2}\nonumber\\
		+&\sum_{i=1}^{n}\bigg\vert \int_{0}^{t} \Big(F_{i}(t, s, \hat{Z}(s))-F_{i}(t, s, z(s))\Big)ds\bigg\vert^{2}\nonumber\\
		+&\bigg\vert \int_{0}^{t} \Big(G_{0}(t, \hat{s}, \hat{Z}(s))-G_{0}(t, s, \hat{Z}(s))\Big)ds\bigg\vert^{2}\nonumber\\
		+&\bigg\vert \int_{0}^{t}\Big( G_{0}(t, s, \hat{Z}(s))-G_{0}(t, s, z(s))\Big)ds\bigg\vert^{2}\nonumber\\
		+&\bigg\vert \int_{0}^{t} \Big(G_{1}(t, \hat{s}, \hat{Z}(s))-G_{1}(t, s, \hat{Z}(s))\Big)ds\bigg\vert^{2}\nonumber\\
		+&\bigg\vert \int_{0}^{t} \Big(G_{1}(t, s, \hat{Z}(s))-G_{1}(t, s, z(s))\Big)ds\bigg\vert^{2}\nonumber\\
		+&\bigg\vert \int_{0}^{t} \Big(G_{2}(t, \hat{s}, \hat{Z}(s))-G_{2}(t, s, \hat{Z}(s))\Big)dW(s)\bigg\vert^{2}\nonumber\\
		+&\bigg\vert \int_{0}^{t} \Big(G_{2}(t, s, \hat{Z}(s))-G_{2}(t, s, z(s))\Big)dW(s)\bigg\vert^{2}\bigg\rbrace\nonumber\\
		=:&(2n+6)\Big\lbrace \hat{H}_{1}+\hat{H}_{2}+\hat{H}_{3}+\hat{H}_{4}+\hat{H}_{5}+\hat{H}_{6}+\hat{H}_{7}+\hat{H}_{8}\Big\rbrace.
	\end{align}
Applying Lemmas 1,2 and 5,  Assumptions 1, 2, 4, and Cauchy–Schwarz's inequality, it can be inferred from the derivation steps of $\hat{K}_{0}$, $\hat{K}_{1}$, $\hat{K}_{2}$, and $\hat{K}_{3}$ (refer to Lemma 6's proof) that
\begin{align}\label{25}
	H_{1}+H_{3}+H_{5}+H_{7}\leq Ch^{2\alpha}.
\end{align}
Using the triangle inequality and the global Lipschitz condition (\ref{4}) along with Cauchy–Schwarz's inequality and Ito isometry, it can be deduced from the steps of (\ref{11}) deduction that
\begin{align}\label{26}
	H_{2}+H_{4}+H_{6}+H_{8}\leq&C^{\prime} K^{2} \bigg\lbrace \sum_{i=1}^{n} \int_{0}^{t}(t-s)^{\alpha_{i}-1}E\Big[\vert \hat{Z}(s)-Z(s)\vert^{2}+\vert Z(t)-z(t)\vert^{2}\Big]ds\nonumber\\
	+&\int_{0}^{t}(t-s)^{\alpha-1}E\Big[\vert \hat{Z}(s)-Z(s)\vert^{2}+\vert Z(t)-z(t)\vert^{2}\Big]ds\bigg\rbrace\nonumber\\
	\leq&C K^{2} \int_{0}^{t}(t-s)^{\alpha-1}E\Big[\vert \hat{Z}(s)-Z(s)\vert^{2}+\vert Z(t)-z(t)\vert^{2}\Big]ds
\end{align}
Finally, applying Lemma 6 and weakly singular Gronwall's inequality, along with combining (\ref{24}-\ref{26}), demonstrates that
\begin{align}
	E\Big[\vert Z(t)-z(t)\vert^{2}\Big]\leq Ch^{2\alpha}
\end{align}
The proof is complete.
\end{proof}
\section{Numerical example} In this part, we investigate the Euler-Maruyama method's convergence rate for various SFNIDEs with weakly singular kernels, as stated in (\ref{18}). We use the sample average to approximate the expectation, as in \cite{12}. More specifically, we estimate the numerical solutions' mean square error by
\begin{align*}
	e_{h,T}=\bigg(\frac{1}{8000}\sum_{i=1}^{8000} \vert Z_{h}(T,\omega_{i})-Z_{\frac{h}{2}}(T,\omega_{i})\vert^{2}\bigg)^{\frac{1}{2}},
\end{align*}
where $\omega_{i}$ denotes the ith single sample path.

\begin{example}
	Consider the $1$-dimensional SFNIDE with $r=1$
	\begin{align*}
		D^{\alpha}\bigg(z(t)-\mathcal{I}^{\alpha_{1}}\cos(tz(t))\bigg)=\cos(tz(t))+\int_{0}^{t} \frac{s\sin(z(s))}{(t-s)^{\beta_{1}}}ds+\int_{0}^{t} \frac{s\sin(z(s))}{(t-s)^{\beta_{2}}}dW(s)
	\end{align*}
for $t\in [0,1]$ and the initial value $z(0)=1$.
\begin{figure}[h]
	\centering
	\includegraphics[width=0.7\textwidth]{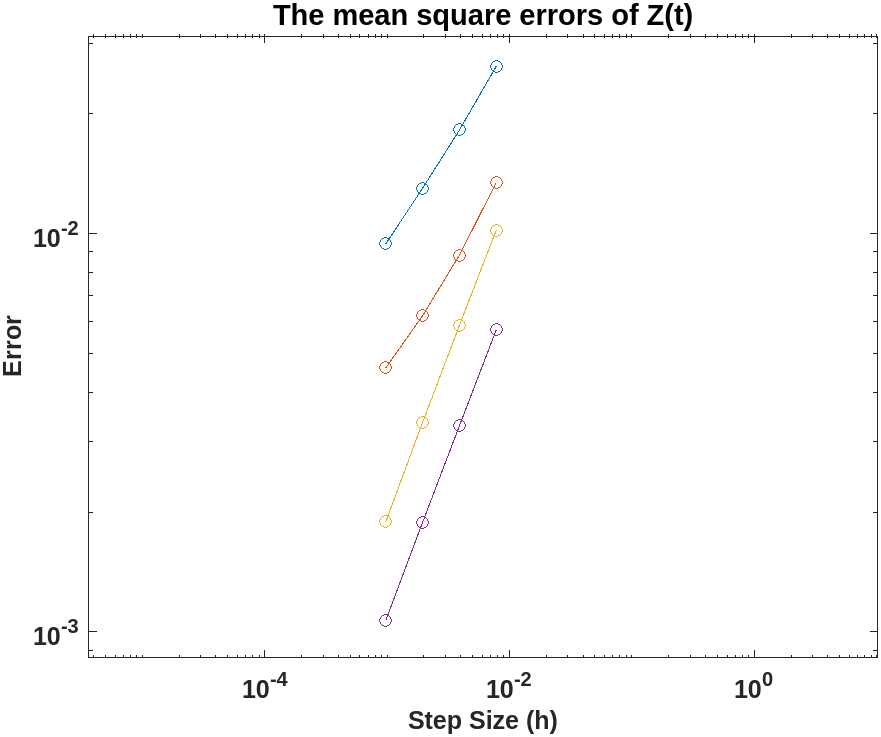}
	\caption{The mean square errors of the EM scheme (\ref{18}) for
		Example 1.\\
		$\color{blue}\boxdot$ $\alpha=0.4, \alpha_{1}=0.5, \beta_{1}=0.6, \beta_{2}=0.4$,\quad
		$\color{red}\boxdot$ $\alpha=0.4, \alpha_{1}=0.5, \beta_{1}=0.8, \beta_{2}=0.3$,\\
		$\color{yellow}\boxdot$ $\alpha=0.8, \alpha_{1}=0.9, \beta_{1}=0.6, \beta_{2}=0.4$,\quad
		$\color{purple}\boxdot$ $\alpha=0.8, \alpha_{1}=0.9, \beta_{1}=0.8, \beta_{2}=0.3$.}\label{fig1}
\end{figure}

We can verify that the functions $g_{i} (i = 0, 1, 2)$ satisfy the hypotheses of  Theorem \ref{t.7}. The computing results are shown in Figure 1. As shown in Figure 1, the
convergence rate is $\alpha$, and the arguments $\alpha$, $\alpha_{1}$, $\beta_{1}$ and $\beta_{2}$ will affect the error constant.
\end{example}
\section{Conclusion}
In summary, this manuscript establishes the existence, uniqueness, and continuous dependence of solutions to nonlinear stochastic fractional neutral integro-differential equations with weakly singular Abel-type kernels, under local Lipschitz and linear growth conditions. Additionally, the Euler-Maruyama method is developed and proven to exhibit strong convergence, in accordance with the well-posedness conditions. Furthermore, the accurate convergence rate of the method is determined under global Lipschitz and linear growth conditions, providing valuable insights into its efficacy for numerical approximation.

\bigskip
\textbf{Acknowledgements} The authors are extremely grateful to Professor Xinjie Dai for providing assistance with the numerical example by sharing his Matlab codes.

\end{document}